\documentclass[11pt]{article}
\usepackage{amsmath,amssymb,amsthm,enumerate}
\usepackage[hmargin=35mm, vmargin=30mm]{geometry}
\usepackage{xy} 
\usepackage{mathrsfs} 
\usepackage{xspace}
\usepackage[colorlinks]{hyperref}

\newtheorem{thm}{Theorem}

\newtheorem{lem}[thm]{Lemma}

\theoremstyle{definition}
\newtheorem*{defn}{Definition}

\newcommand{\mc}[1]{\mathcal{#1}}

\newcommand{\mf}[1]{\mathfrak{#1}}

\newcommand{\tdlc}{t.d.l.c.\@\xspace}

\newcommand{\Mon}{\mathrm{Mon}}

\newcommand{\inv}{^{-1}}
\newcommand{\triv}{\{1\}}
\newcommand{\QZ}{\mathrm{QZ}}
\newcommand{\QC}{\mathrm{QC}}
\newcommand{\CC}{\mathrm{C}}
\newcommand{\N}{\mathrm{N}}
\newcommand{\con}{\mathrm{con}}

\newcommand{\grp}[1]{\langle #1 \rangle}

\usepackage[usenames, dvipsnames]{color}

\newcommand{\defbold}{\textbf}

\begin{document}

\title{A sufficient condition for a locally compact almost simple group to have open monolith}
\author{Colin D. Reid}

\maketitle

\begin{abstract}
We obtain a sufficient condition, given a totally disconnected, locally compact group $G$ with a topologically simple monolith $S$, to ensure that $S$ is open in $G$ and abstractly simple.
\end{abstract}

\paragraph{\textbf{Acknowledgement}}The impetus for writing this note was a question of Waltraud Lederle, as well as some questions arising from \cite{CRW} and from an ongoing project with Alejandra Garrido and David Robertson.  I thank Waltraud Lederle and my collaborators for their insightful questions and comments.

\

We recall some definitions from \cite{CRW}.

\begin{defn}
Let $G$ be a totally disconnected, locally compact (\tdlc) group.  We say $G$ is \defbold{expansive} if there is a neighbourhood $U$ of the identity in $G$ such that $\bigcap_{g \in G}gUg\inv = \triv$.  The group $G$ is \defbold{regionally expansive} if there is a compactly generated open subgroup $O$ of $G$ such that $O$ is expansive; equivalently, every open subgroup containing $O$ is expansive.

A topological group $G$ is \defbold{monolithic} if there is a unique smallest nontrivial closed normal subgroup of $G$, called the \defbold{monolith} $\Mon(G)$ of $G$.  A \tdlc group $G$ is \defbold{robustly monolithic} if it is monolithic and the monolith is nondiscrete, regionally expansive, and topologically simple.
\end{defn}

Note that every topologically simple \tdlc group $S$ is expansive, and hence if $S$ is compactly generated, then it is regionally expansive.  Thus in the context of topologically simple groups, `regionally expansive' should be considered a generalization of `compactly generated'.  The definition of `robustly monolithic' allows us to consider a more general situation, where the regionally expansive topologically simple group $S$ is embedded as a closed normal subgroup in some larger \tdlc group $G$, such that $\CC_G(S)=\triv$.  It is then natural to ask how complex the quotient $G/S$ can be as a topological group.

Here is a sufficient condition for $G/S$ to be discrete, in other words, for $S$ to be open in $G$; in the situation described, in fact $S$ is abstractly simple.

\begin{thm}\label{mainthm}
Let $G$ be a robustly monolithic \tdlc group.  Suppose that $G$ has an open subgroup of the form $K \times L$ where $K$ and $L$ are nontrivial closed subgroups of $G$.  Then for every nontrivial subgroup $H$ of $G$ such that $\Mon(G) \le \N_G(H)$, without assuming that $H$ is closed, it follows that $H$ is open in $G$ and contains $\Mon(G)$.  In particular, $\Mon(G)$ itself is abstractly simple and open in $G$.
\end{thm}

The proof is based on the local structure theory developed in \cite{CRW-Part1}, \cite{CRW-Part2} and \cite{CRW}; we briefly recall the necessary background.

\begin{defn}
We define the \defbold{quasi-centralizer} $\QC_G(H)$ of a subgroup $H$ of a topological group $G$ to be the set of elements $g \in G$ such that $g$ commutes with some open subgroup of $H$, and write $\QZ(G) := \QC_G(G)$.  We say $G$ is \defbold{[A]-semisimple} if $\QZ(G) = \triv$, and whenever $A$ is an abelian subgroup of $G$ with open normalizer, then $A = \triv$.

Given an [A]-semisimple \tdlc group $G$, the \defbold{(globally defined) centralizer lattice} of $G$ is the set
\[
\mathrm{LC}(G) = \{\CC_G(K) \mid K \le G, \N_G(K) \text{ is open in } G\},
\]
equipped with the partial order of inclusion of subsets of $G$.  Within $\mathrm{LC}(G)$, the \defbold{(globally defined) decomposition lattice} $\mathrm{LD}(G)$ consists of those $K \in \mathrm{LC}(G)$ such that $K\CC_G(K)$ is open in $G$.
\end{defn}

By {\cite[Proposition~5.1.2]{CRW}}, every robustly monolithic \tdlc group $G$ is [A]-semisimple, so the definitions above apply.  Note that if $G$ is [A]-semisimple, then so is every open subgroup of $G$.

By construction, given $K \in \mathrm{LC}(G)$, then $K$ is closed in $G$ and $\N_G(K)$ is open in $G$.  It is shown in \cite{CRW-Part1} that $\mathrm{LC}(G)$ is a Boolean algebra, on which the map $K \mapsto \CC_G(K)$ is the complementation map.  The centralizer lattice is a local invariant of $G$, in the sense that if $O$ is any open subgroup of $G$, then $\mathrm{LC}(G)$ is $O$-equivariantly isomorphic to $\mathrm{LC}(O)$ via the map $K \mapsto K \cap O$.  The decomposition lattice is a local invariant of $G$ in the same manner; in particular, it accounts for all direct factors of open subgroups of $G$.  The decomposition lattice has the following additional property:

$(*)$ Given $A_1,\dots,A_n \in \mathrm{LD}(G)$ with least upper bound $A$ in $\mathrm{LD}(G)$, and given open subgroups $B_i$ of $A_i$, then as a subset of $A$, the product $B_1 B_2 \dots B_n$ is a neighbourhood of the identity.

(To see why $(*)$ holds, note that we can choose compact open subgroups $B_i$ of $A_i$ that normalize each other, so that $B = B_1 B_2 \dots B_n$ is a compact subgroup of $G$; the fact that $B$ is an open subgroup of an element of $\mathrm{LD}(G)$ then follows by \cite[Theorem~4.5]{CRW-Part1}.)

There is a natural action of $G$ on $\mathrm{LC}(G)$ by conjugation, which preserves the partial order and hence the Boolean algebra structure; note also that the stabilizers of this action are open.  There is then a corresponding continuous action of $G$ by homeomorphisms on the Stone space $\mf{S}(\mathrm{LC}(G))$ of $\mathrm{LC}(G)$, which is a compact zero-dimensional Hausdorff space.  The latter action has useful dynamical properties.  Given a group $G$ acting on a topological space $X$, we say the action is \defbold{minimal} if every orbit is dense, and \defbold{compressible} if there is a nonempty open subset $Y$ such that for every nonempty open subset $Z$ of $X$, there is $g \in G$ such that $gY \subseteq Z$.

\begin{lem}\label{lem:simple_dynamics}
Let $G$ be a robustly monolithic \tdlc group and let $\mc{A}$ be a $G$-invariant subalgebra of $\mathrm{LC}(G)$.  Then the $G$-action on $\mf{S}(\mc{A})$ is continuous and the $\Mon(G)$-action is minimal and compressible.
\end{lem}

\begin{proof}
The action is continuous because stabilizers of elements of $\mc{A}$ are open.  By {\cite[Theorem~7.3.3]{CRW}}, the action of $\Mon(G)$ on $\mf{S}(\mathrm{LC}(G))$ is minimal and compressible; these properties pass to any quotient $\Mon(G)$-space.
\end{proof}

We can now prove the theorem.

\begin{proof}[Proof of Theorem~\ref{mainthm}]
As noted above, $G$ is [A]-semisimple, and hence the centralizer lattice $\mathrm{LC}(G)$ is a Boolean algebra, with $G$-invariant subalgebra $\mathrm{LD}(G)$.  The condition that $G$ has an open subgroup that splits nontrivially as a direct product then amounts to the condition that $|\mathrm{LD}(G)| > 2$.  Applying Stone duality, the $G$-space $X_C := \mf{S}(\mathrm{LC}(G))$ admits a $G$-equivariant quotient space $X_D := \mf{S}(\mathrm{LD}(G))$ with $|X_D|>1$.

Let $S = \Mon(G)$.  By Lemma~\ref{lem:simple_dynamics}, the action of $G$ on $X_D$ is continuous and the action of $S$ is minimal and compressible.  In particular, $S$ acts nontrivially on $X_D$.  Since $S$ is topologically simple and the action is continuous, in fact $S$ acts faithfully on $X_D$, so the action of $G$ on $X_D$ is also faithful.

\emph{Claim: Let $H$ be a nontrivial subgroup of $G$, such that $\N_G(H)$ is open in $G$ and contains $S$ (do not assume that $H$ is closed).  Then $H$ is open in $G$ and $S \le H$.}

We begin the proof of the claim with some reductions.  By \cite[Lemma~5.1.4]{CRW}, the normalizer $\N_G(H)$ is robustly monolithic with monolith $S$.  Since $\N_G(H)$ is open in $G$, it also has an open subgroup that splits nontrivially as a direct product.  Thus we may assume without loss of generality that $G = \N_G(H)$.  Since $H$ and $S$ are normal, we have $[H,S] \le H \cap S$.  Since $\CC_G(S) = \triv$, it follows that $H$ and $S$ have nontrivial intersection; thus we may replace $H$ with $H \cap S$ and assume that $H$ is a normal subgroup of $S$.  Since $S$ is topologically simple, it follows that $H$ is dense in $S$.  By continuity, we then see that the action of $H$ on $X_D$ is minimal and compressible.

Since $|X_D|>1$, $X_D$ is zero-dimensional, and the action is compressible, there is a nonempty clopen subset $\alpha$ of $X$ and $h \in H$ such that $h\alpha$ is properly contained in $\alpha$.  Correspondingly, there is $K \in \mathrm{LD}(G)$ and $h \in H$, such that $hKh\inv$ is a subgroup of $K$ that is closed but not open in $K$.  It then follows that $K$ has an open subgroup of the form $hKh\inv \times L$ where $L = \CC_K(hKh\inv)$; in turn, $L$ is a nontrivial element of $\mathrm{LD}(G)$.  By \cite[Proposition~5.1]{CRW-TitsCore}, we have $\con_G(h) \le H$, where
\[
\con_G(h) := \{g \in G \mid h^ngh^{-n} \rightarrow 1 \text{ as } n \rightarrow +\infty\}.
\]
By \cite[Proposition~6.14]{CRW-Part2}, the intersection $L^* = \con_G(h) \cap L$ is open in $L$.  Let $\beta$ be the clopen subset of $X_D$ corresponding to $L \in \mathrm{LD}(G)$.  Since $X_D$ is compact and the action of $H$ on $X_D$ is minimal, there are $h_1,\dots,h_n \in H$ such that $X_D = \bigcup^n_{i=1}h_i\beta$.  By $(*)$ it follows that $\grp{h_iL^*h\inv_i \mid 1 \le i \le n}$ is open in $G$; thus $H$ is open in $G$.  In particular, $H$ is closed in $S$; since $H$ is dense in $S$, it follows that $H = S$.  This completes the proof of the claim.

The claim applies in particular when $S = H$, so $S$ is open in $G$.  Thus the assumption that $\N_G(H)$ is open follows automatically from assuming that $\N_G(H)$ contains $S$.  We have therefore shown that given a nontrivial subgroup $H$ of $G$ such that $S \le \N_G(H)$, then $H$ is open in $G$ and $S \le H$.  As this conclusion applies in particular to any nontrivial normal subgroup of $S$, we see that $S$ is abstractly simple.
\end{proof}

\end{document}